\documentclass[microtype]{gtpart}
\usepackage[parfill]{parskip}
\usepackage{graphicx} 
\usepackage{amsmath}
\usepackage{amstext}
\usepackage{amsfonts}
\usepackage{amsthm}
\usepackage{amssymb}
\usepackage{amscd}
\usepackage{subfig}
\usepackage{hyperref}
\usepackage[matrix]{xy}

\usepackage{mathrsfs}
\usepackage{mathtools}
\usepackage{comment}
\usepackage{tikz}
\usetikzlibrary{trees}
\usetikzlibrary{arrows}
\usetikzlibrary{matrix}
\usepackage{ifthen}

\title{Nonf\hspace{0.02cm}illable Legendrian knots in the 3-sphere}
\author{\ \ \ Tolga Etg\"u} 
\address{Department of Mathematics, Ko\c{c} University, Sariyer, Istanbul 34450 Turkey \newline
Department of Mathematics, Princeton University, Princeton, NJ 08540 USA}
\email{tetgu@math.princeton.edu}\email{tetgu@ku.edu.tr}

\newtheorem{theorem}{Theorem}
\newtheorem{lemma}[theorem]{Lemma}
\newtheorem{proposition}[theorem]{Proposition}

\newtheorem{corollary}[theorem]{Corollary}

\newtheorem{remark}[theorem]{Remark}

\renewcommand{\d}{\partial} 

\newcommand{\K}{\mathbb{K}}

\newcommand{\ep}{\epsilon}

\newcommand{\Aff}{\mathcal{A}}
\newcommand{\Bff}{\mathcal{B}}
\newcommand{\Cff}{\mathcal{C}}
\newcommand{\Fff}{\mathcal{F}}

\begin{document}

\maketitle

\begin{abstract}
Let $\Lambda$ be a Legendrian knot in the standard contact 3-sphere. 
If $\Lambda$ bounds an orientable exact Lagrangian surface $\Sigma$ in the standard symplectic 4-ball, then the genus of $S$ is equal to the slice genus of (the smooth knot underlying) $\Lambda$, the rotation number of $\Lambda$ is zero as well as the sum of the Thurston-Bennequin number of $\Lambda$ and the Euler characteristic of $\Sigma$, and moreover the linearized contact homology of $\Lambda$ with respect to the augmentation induced by $\Sigma$ is isomorphic to the (singular) homology of $\Sigma$. 
It was asked in \cite{EHK} whether the converse of this statement holds. 
We give a negative answer to this question providing a family of Legendrian knots with augmentations which are not induced by exact Lagrangian fillings although the associated linearized contact homology is isomorphic to the homology of the smooth surface of minimal genus in the 4-ball bounding the knot.
\end{abstract}

\section{Introduction}

Let $\Lambda$ be a Legendrian knot in the standard contact 3-sphere. 
If $\Lambda$ bounds an  exact orientable Lagrangian surface $\Sigma$ in the standard symplectic 4-ball, then the genus of $\Sigma$ is equal to the slice genus of (the smooth knot underlying) $\Lambda$, the rotation number of $\Lambda$ is zero as well as the sum of the Thurston-Bennequin number of $\Lambda$ and the Euler characteristic of $\Sigma$ by a theorem of Chantraine \cite{cha}. Moreover the linearized contact homology of $\Lambda$ with respect to the augmentation induced by $\Sigma$ is isomorphic to the (singular) homology of $\Sigma$ by a theorem of Seidel \cite{E, D-R}. 
Question (8.9) in \cite{EHK} asks whether every augmentation for which Seidel's isomorphism holds is induced by a Lagrangian filling.
We give a negative answer to this question based on
the family of Legendrian knots
$$\{ \Lambda_{p,q,r,s} : p,q,r,s \geq 2,\  p \equiv q \equiv r+1 \equiv s+1 \mod 2 \}$$
given by the  Lagrangian  projection in Fig.~(\ref{lag}). 
Throughout the paper we will always be working under the above assumption that the parities of $p$ and $q$ match and they are the opposite of those of $r$ and $s$.
The rotation and Thurston-Bennequin numbers of $\Lambda_{p,q,r,s}$ are 0 and 5, respectively. 
This gives a lower bound of $3$ on the slice genus. 
On the other hand, the Seifert surface we obtain from this projection has genus $3$ hence both the slice genus and the Seifert genus are $3$.
\begin{theorem}
The Chekanov-Eliashberg dg-algebra of $ \Lambda_{p,p,p+1,p+1}$ admits an augmentation which is not induced by any exact orientable Lagrangian filling, although the corresponding linearized contact homology is isomorphic to the  homology of a surface of genus $3$ with one boundary component.
\end{theorem}
Our examples are inspired by a deformation argument in \cite{EL} related to the Chekanov-Eliashberg algebra where the Legendrian link of unknots linked according to the $D_4$ tree is shown to be significant, specifically over a base field of characteristic $2$ (see the last part of the proof of Thm.(14) in \cite{EL} and Rem.~(\ref{intr}) below). 
As can be seen in Fig.~(\ref{lag}), $\Lambda_{p,q,r,s}$ is constructed from the $D_4$ link by adding twists that turn it into a Legendrian knot. 
A similar construction was previously used in \cite{CEKSW} on a different link to produce infinite families of Legendrian knots not isotopic to their Legendrian mirrors.

In the next section we describe the Chekanov-Eliashberg algebra and compute the linearized contact homology of our examples. 
In Sec.~(\ref{dualizing}) we gather enough information on the strictly unital $A_\infty$-algebras obtained by dualizing the Chekanov-Eliashberg algebras of our Legendrian knots to prove that they are not quasi-isomorphic to the $A_\infty$-algebra of cochains on a closed surface whenever the base field has characteristic $2$.  This proves the nonfillability of our examples by a duality result of Ekholm and Lekili from their recent preprint \cite{EkL} (see Thm.~(\ref{ekhlek}) below).

{\bf Acknowledgments.} We would like to thank Yank\i \ Lekili and Joshua Sabloff for discussions on the subject and comments on a draft of this paper. It is a pleasure to thank Princeton University for the  hospitality. This research is partially supported by the Technological and Research Council of Turkey through a BIDEB-2219 fellowship.

\section{Linearized contact homology  of $\Lambda_{p,q,r,s}$}

The Chekanov-Eliashberg algebra of a Legendrian knot is a differential graded algebra generated by Reeb chords from the knot to itself.
We refer to \cite{chekanov} for the combinatorial description of the Chekanov-Eliashberg algebra of a Legendrian knot in the standard contact $3$-sphere. 

We denote the Chekanov-Eliashberg algebra of $\Lambda= \Lambda_{p,q,r,s}$ over a field $\K$ by $(L, \d )$. 
Since the rotation number of $\Lambda$ is $0$, we have a $\Z$-grading on $(L, \d )$. 

\begin{figure}[htb!]
\centering
\begin{tikzpicture}
	\draw (0,0) node[inner sep=0] {\includegraphics[scale=0.5]{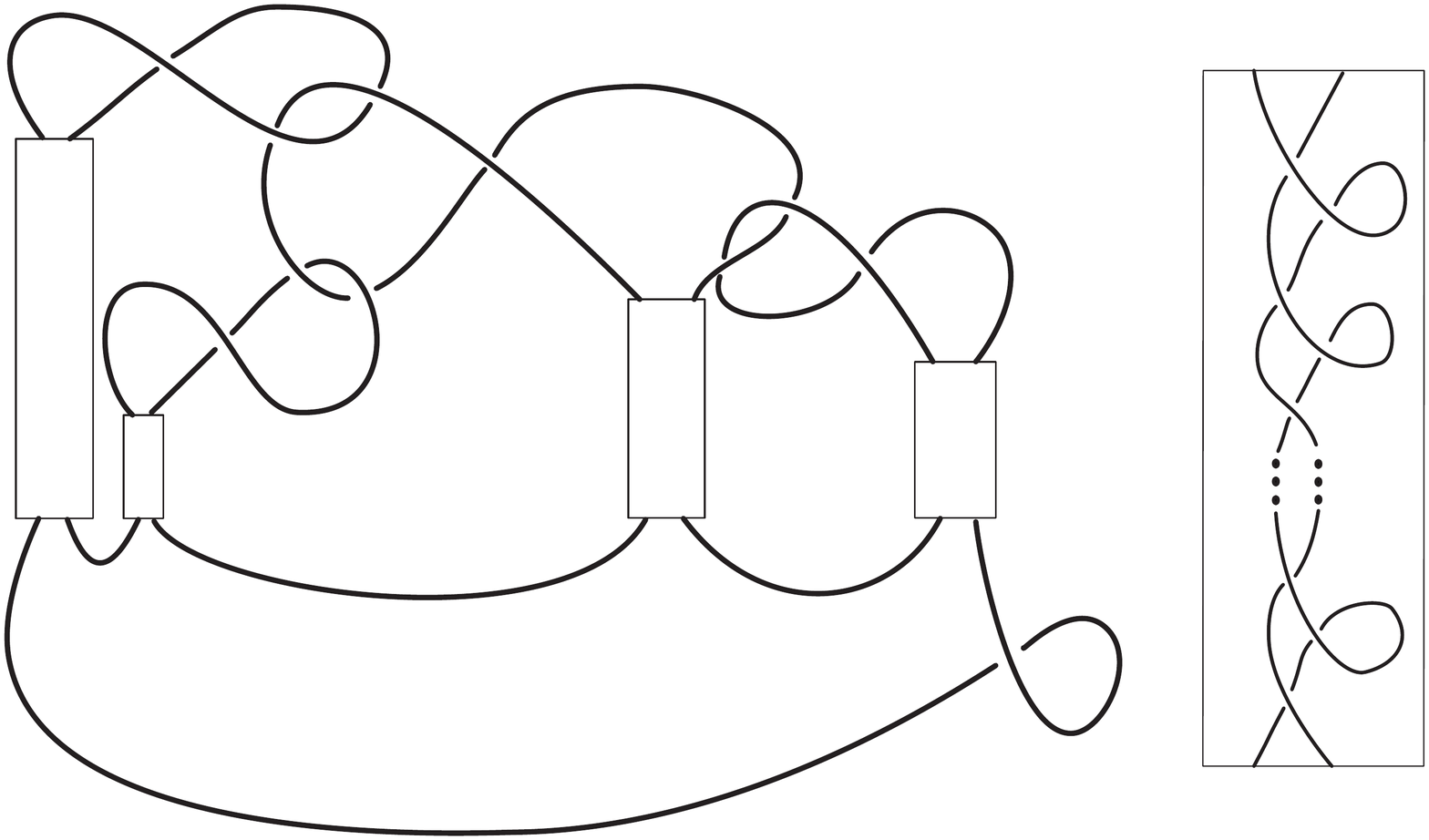}} ;
	\node[] at (-6.4,1.3)  {{{$\lambda^x_p$}}};
	\node[] at (-5.53,0.27)  {{{$\lambda^y_q$}}};
	\node[] at (-0.4,0.85)  {{{$\lambda^z_r$}}};
	\node[] at (2.5,0.65)  {{{$\lambda^w_s$}}};
	\node[] at (4.5,1.5)  {{\large{$\lambda^*_n$}:}};
	\node[] at (-5.3,4.5)  {{\tiny{$a^x_0$}}};
	\node[] at (-4.5,3.5)  {{\tiny{$b_1$}}};
	\node[] at (-2.9,4.0)  {{\tiny{$b_4$}}};
	\node[] at (-4.7,1.75)  {{\tiny{$a^y_0$}}};
	\node[] at (-4.25,2.2)  {{\tiny{$b_5$}}};
	\node[] at (-3.05,1.85)  {{\tiny{$b_2$}}};
	\node[] at (-2.15,3.45)  {{\tiny{$a^z_0$}}};
	\node[] at (0.0,2.25)  {{\tiny{$b_6$}}};
	\node[] at (1.15,2.82)  {{\tiny{$b_3$}}};
	\node[] at (1.55,2.5)  {{\tiny{$a^w_0$}}};
	\node[] at (3.15,-1.45)  {{\tiny{$a_0$}}};
	\node[] at (5.62,1.62+1.6)  {{\tiny{$*_0$}}};
	\node[] at (5.55,0.27+1.6)  {{\tiny{$*_1$}}};
	\node[] at (5.67,-0.88+1.6)  {{\tiny{$*_2$}}};
	\node[] at (5.47,-2.5+1.6)  {{\tiny{$*_{n-1}$}}};
	\node[] at (5.57,-3.75+1.6)  {{\tiny{$*_{n}$}}};
	\node[] at (6.43,1.12+1.62)  {{\tiny{$a^*_1$}}};
	\node[] at (6.42,-0.23+1.62)  {{\tiny{$a^*_2$}}};
	\node[] at (6.34,-3.12+1.6)  {{\tiny{$a^*_n$}}};
	\node[] at (-2.66,-3.47)  {{\Large{$\bullet$}}};
	
\end{tikzpicture}
\vspace{-1.6cm}
\caption{Lagrangian projection of $\Lambda_{p,q,r,s}$ for $p,q,r,s \geq 2$ and $p \equiv q \equiv r+1 \equiv s+1 \mod 2$}
\label{lag}
\end{figure}

The generators of $(L, \d )$ are as indicated in Fig.~(\ref{lag}):  
\begin{align*}
 a^x_0, \dots , a_p^x , \ a^y_0,  \dots , a_q^y, & \ a^z_0, \dots , a_r^z, \ a^w_0, \dots , a_s^w \\
x_0 , \dots , x_p, \  y_0 , \dots , y_q,  & \ z_0 , \dots , z_r, \ \ w_0 , \dots ,  w_s \\
a_0 , b_1, & \dots , b_6 
\end{align*}
with gradings
\begin{align*}
&  |*_i|=0, |a_0| =|a^*_i|=1 , \ \mbox{ for } \ *  \in \{ x,y,z,w \} \\
|b_1|=-|b_4| &= p-r+1, \ |b_2|=-|b_5| = q-r+1, \ |b_3|=-|b_6| = r-s 
\end{align*}
Let $\ep : L \to \K $ be the augmentation which maps all $*_i$ to $-1$ for $*  \in \{ x,y,z,w \}$.
\begin{proposition}\label{ce-dga}
There is an isomorphism 
$$HC_*^\ep (\Lambda_{p,p,p+1,p+1}) \cong  H_{1-*} (\Sigma ; \K)$$
between the linearized contact homology of $\Lambda_{p,p,p+1,p+1}$ with respect to the augmentation $\ep$ and the homology of the orientable surface $\Sigma$ of genus $3$ with one boundary component.
\end{proposition}
\begin{proof}
Counting the relevant immersed polygons with the choice of a base-point on $\Lambda_{p,q,r,s}$ as indicated by {\Large{$\bullet$}} in Fig.~(\ref{lag}) we see that nontrivial differentials are given by
\begin{align*}
\d a_0 & = 1 - w_sz_ry_qx_p \\
\d a^x_0 & = 1 + x_0 + b_1b_4 \\
\d a^y_0 & = 1 + y_0 + b_2b_5 \\
\d a^z_0 & = 1 + z_0 + b_4b_1+ b_5b_2 + z_0 b_6b_3 +  b_4b_1 b_5b_2  \\
\d a^w_0 & = 1 + w_0 + b_3b_6 \\
\d a_i^* & = 1 - *_{i-1}*_i 
\mbox{ for }  \ *  \in \{ x,y,z,w \} \ \mbox{ and } i \geq 1
\end{align*} 
Conjugating $\d$ by the automorphism $id+\ep$ gives another differential $\d^\ep$ on $L$ 
\begin{align*}
\d^\ep a_0 & = w_s + z_r + y_q+ x_p - w_sz_r- w_sy_q - w_sx_p - z_ry_q - z_sx_p - y_qx_p \\ & + w_sz_ry_q+ w_sz_rx_p+w_sy_qx_p+z_ry_qx_p- w_sz_ry_qx_p \\
\d^\ep a^x_0 & =  x_0 + b_1b_4 , \ \ \d^\ep a^y_0 = y_0 + b_2b_5 , \ \ \d^\ep a^w_0  =  w_0 + b_3b_6 \\
\d^\ep a^z_0 & =z_0 + b_4b_1+ b_5b_2 - b_6b_3 + z_0 b_6b_3 +  b_4b_1 b_5b_2  \\
\d^\ep a_i^* & =  *_{i-1}+*_i-*_{i-1}*_i \ , \ 
\mbox{ for }  \ *  \in \{ x,y,z,w \} \ \mbox{ and } i \geq 1
\end{align*}

Applying the elementary transformation 
$$ a_0 \mapsto  a_0 -(-1)^{p} \left(\sum_{i=0}^p (-1)^i a_i^x +  \sum_{i=0}^q (-1)^i a_i^y
- \sum_{i=0}^r (-1)^i a_i^z- \sum_{i=0}^s (-1)^i a_i^w\right) $$
simplifies the computation of linearized contact homology $HC_*^\ep(\Lambda)$ of $\Lambda$ associated to the augmentation $\ep$ and more importantly, the description of the $A_\infty$-algebras that will be discussed in the next section. 

At this point, we have the following presentation of the differential $\d_1^\ep$ on the linearized complex which computes $HC_*^\ep(\Lambda_{p,q,r,s})$:
\begin{align*}
&\d_1^\ep a_0  = \d_1^\ep b_j  = 0  \\
\d_1^\ep a^x_0  =  x_0  , \ \
\d_1^\ep a^y_0 &= y_0,   \ \
\d_1^\ep a^z_0  =z_0,    \ \
\d_1^\ep a^w_0  =  w_0  \\
\d_1^\ep a_i^*  =  *_{i-1}+*_i & \ , \ 
\mbox{ for }  \ *  \in \{ x,y,z,w \} \ \mbox{ and } i \geq 1
\end{align*}
It is clear that $HC_*^\ep(\Lambda_{p,q,r,s})$ is spanned by $a_0 , b_1, \dots , b_6$. 
Moreover, if $p=q=r-1=s-1$, then $|b_i|=0$ for all $i$ and we get the graded isomorphism in the statement.
\end{proof}

\section{The augmentation $\ep$ is not induced by a Lagrangian filling}\label{dualizing}

In this section, we prove that $\Lambda_{p,q,r,s}$ has no exact Lagrangian filling associated to the augmentation $\ep$ by using a result from a recent preprint of Ekholm and Lekili. 

The following is a consequence of  \cite[Thm.(4)]{EkL}.
\begin{theorem}[Ekholm-Lekili]\label{ekhlek}
If $\Lambda$ has an exact Lagrangian filling $\Sigma$, then there is an $A_\infty$ quasi-isomorphism between $\operatorname{RHom}_{CE^*}(\K, \K)$ and the $A_\infty$-algebra $C^*(S;\K)$ of (singular) cochains on the closed surface $S$ obtained by capping the boundary of $\Sigma$, where $CE^*=L_{-*}$ and $\K$ is equipped with a $CE$-module structure by the augmentation $\ep_\Sigma : CE \to \K$ induced by the filling $\Sigma$.
\end{theorem}

In order to describe the $A_{\infty}$-algebra $\operatorname{RHom}_{CE^*}(\K, \K)$ in the above statement, we  utilize the isomorphism between $\operatorname{RHom}_{CE^*}(\K, \K)$ and the linear dual $\Bff$ of the  Legendrian $A_\infty$-coalgebra defined in the more general setting of \cite{EkL}. 
In the current setup, the strictly unital $A_\infty$-algebra $\Bff$ can be obtained from the non-unital $A_\infty$-algebra on the linearized cochain complex defined in \cite{CEKSW} (which is also the endomorphism algebra of $\ep$ in the $Aug_-$ category of \cite{BC}) by adding a copy of $\K$ to make it unital (cf. \cite[Rem.(24)]{EkL}). 

The description of $\Bff\cong \operatorname{RHom}_{CE^*}(\K, \K)$ we provide is based on the presentation of $(L, \d^\ep)$ obtained at the end of the proof of Prop.~(\ref{ce-dga}).
We abuse the notation and denote the duals of the generators of $L$ by the generators themselves.
The nontrivial $A_\infty$-products on $\Bff $ (besides those dictated by strict unitality)  are
$$ \mu_{\Bff}^1 (*_i)  =a_i^* + a_{i+1}^* \ , \  \mbox{ for }  \ *  \in \{ x,y,z,w \} \ \mbox{ and } i \geq 0$$
$$\mu_{\Bff}^1 (x_p)  =a_p^x  \ , \ \mu_{\Bff}^1 (y_q)  =a_q^y  \ , \ \mu_{\Bff}^1 (z_r)  = a_r^z  \ , \ \mu_{\Bff}^1 (w_s)  =a_s^w $$
$$\mu_{\Bff}^2 (b_1,b_4) = (-1)^{p+1}a_0+a^x_0 \ , \  \mu_{\Bff}^2 (b_2,b_5) =(-1)^{p+1} a_0+a^y_0 \ , \  \mu_{\Bff}^2 (b_3,b_6) = (-1)^{p}a_0+a^w_0 $$
$$ \mu_{\Bff}^2 (b_4,b_1) = \mu_{\Bff}^2 (b_5,b_2) = (-1)^{p}a_0+a^z_0 \ , \  \mu_{\Bff}^2 (b_6,b_3) =(-1)^{p+1}a_0-a^z_0 $$
$$ \mu_{\Bff}^2 (x_{i-1}, x_i) = (-1)^{p+i}a_0 - a_i^x  ,  \ \ \mu_{\Bff}^2 (y_{i-1}, y_i) = (-1)^{p+i}a_0 - a_i^y $$
$$  \mu_{\Bff}^2 (z_{i-1}, z_i) = (-1)^{p+i+1}a_0 - a_i^z,  \ \ \mu_{\Bff}^2 (w_{i-1}, w_i) = (-1)^{p+i+1}a_0 - a_i^w    $$ 
$$ \mu_{\Bff}^2 (w_s,z_r) = \mu_{\Bff}^2 (w_s,y_q) = \mu_{\Bff}^2 (w_s,x_p) = \mu_{\Bff}^2 (z_r,y_q)  = \mu_{\Bff}^2 (z_r,x_p)  =\mu_{\Bff}^2 (y_q,x_p) = - a_0 $$
$$  \mu_{\Bff}^3 (z_0 , b_6,b_3) = a_0+a^z_0 $$
$$ \mu_{\Bff}^3 (w_s,z_r,y_q) = \mu_{\Bff}^3 (w_s,z_r,x_p) =  \mu_{\Bff}^3 (w_s,y_q,x_p) = \mu_{\Bff}^3 (z_r,y_q,x_p) = a_0 $$
$$ \mu_{\Bff}^4 (b_4,b_1,b_5,b_2) = a_0+a^z_0 \ , \ \mu_{\Bff}^4 (w_s,z_r,y_q,x_p) = - a_0 $$
and the gradings in $\Bff$ are 
$$|a_0| =|a^*_i|=2 , \ |*_i|=1 \mbox{ for }  \ *  \in \{ x,y,z,w \} \ \mbox{ and } i \geq 0$$
$$|b_1|=-|b_4| = p-r+2, \ |b_2|=-|b_5| = q-r+2, \ |b_3|=-|b_6| = r-s+1 $$
A homological perturbation argument, as suggested by Prop.~(1.12) and Rem.~(1.13) in \cite{Seidel}, provides a minimal model $(\Aff, \mu^\bullet_\Aff)$ quasi-isomorphic to $(\Bff, \mu^\bullet_\Bff)$. 
From the above description of the $A_\infty$-products we see the decomposition $\Bff=\Aff \oplus \Cff$, where
$\Aff$ is generated by 
$$\{ 1, a_0 , b_i : i=1,\dots , 6\}$$ 
and gives a minimal model for $\Bff$, whereas $\Cff$ is the subalgebra generated by the rest of the generators of $\Bff$ and acyclic with respect to the differential $\mu^1$.  
We choose a contracting homotopy $T^1: \Cff \to \Cff$ with $\mu^1_\Bff T^1 + T^1 \mu^1_\Bff = F^1G^1 - id$ as follows
$$ T^1 (a_i^x) = -x_i +x_{i+1}+ \cdots +(-1)^{p-i+1} x_p , \ T^1 (a_i^y) =- y_i +y_{i+1}+  \cdots + (-1)^{q-i+1}y_q  $$$$ T^1 (a_i^z) = -z_i +z_{i+1}+ \cdots +(-1)^{r-i+1} z_r , \ T^1 (a_i^w) = - w_i +w_{i+1}+ \cdots + (-1)^{s-i+1}w_s$$
This homotopy $T^1$ and Eqn.~(1.18) in \cite{Seidel} suffices to compute the $A_\infty$-products $\mu^\bullet_\Aff$ on $\Aff$. 
To begin with, the only nontrivial $\mu^2_\Aff$ products (besides those dictated by strict unitality) are
$$\mu_\Aff^2(b_1,b_4)=\mu_\Aff^2(b_2,b_5)=\mu_\Aff^2(b_6,b_3)=(-1)^{p+1}a_0$$
$$\mu_\Aff^2(b_4,b_1)=\mu_\Aff^2(b_5,b_2) = \mu_\Aff^2(b_3,b_6)=(-1)^{p}a_0$$
\newpage
Moreover, $\mu^3_\Aff$ vanishes since 
\begin{itemize}
\item $\mu^3_\Bff$ is trivial on $\Aff \otimes \Aff \otimes \Aff$, and 
\item $ \mu^2_\Bff$ vanishes on $im(F^2) \otimes \Aff \subset im(T^1) \otimes \Aff \subset \Cff \otimes \Aff$  and on $\Aff \otimes im(F^2) \subset \Aff \otimes   \Cff $
\end{itemize} 
where $F^2 = T^1 \circ \mu^2_\Bff$ \ .

Proceeding further, Eqn.~(1.18) in \cite{Seidel} for $d=4$ simplifies as
$$
\mu^4_\Aff (\alpha_4, \dots , \alpha_1) = G^1 \left(\mu^4_\Bff (\alpha_4, \dots , \alpha_1)+ \mu^3_\Bff \left(F^2 (\alpha_4, \alpha_3) , \alpha_2 , \alpha_1\right)+\mu^2_\Bff \left(F^2 (\alpha_4, a_3) , F^2 (\alpha_2 , \alpha_1)\right)\right)
$$
where $G^1 : \Bff \to \Aff$ is the projection. 

At this point, one immediately sees that all three summands on the right hand side of the above formula for $\mu^4_\Aff$ vanish on quadruples which are not of the form $(b_i,b_{i\pm3},b_j,b_{j\pm 3})$.  
In order to prove the key proposition below, it suffices to compute $\mu^4_\Aff(b_i,b_j,b_k,b_l)$, where $(b_i,b_j,b_k,b_l)$ is a cyclic permutation of  $(b_1,b_2,b_5,b_4)$,  $(b_1,b_5,b_2,b_4)$, or $(b_2,b_5,b_2,b_5)$.
To this end, it is straightforward, if tedious, to check that, depending on the parity of $p$, the following are the only nontrivial ones among these ten $\mu^4$ products:
$$  \mu^4_\Aff(b_4,b_1,b_2,b_5)=a_0, \ \mu^4_\Aff(b_5,b_2,b_4,b_1)=\mu^4_\Aff(b_5,b_2,b_5,b_2)=-a_0 $$
if $ p$ is even, and
$$  \mu^4_\Aff(b_4,b_1,b_2,b_5)=\mu^4_\Aff(b_4,b_1, b_5, b_2)=a_0, \ \mu^4_\Aff(b_2,b_5,b_2, b_5)=-a_0 $$
if $p$ is odd.

Besides the computations above, another important ingredient for the proof of the proposition below is the following formality statement.
\begin{lemma}\label{formal}
Over a field $\K$ of characteristic $2$, the algebra of cochains $C^*(S;\K)$ on a closed orientable surface $S$ is a formal differential graded algebra.
\end{lemma}
Over the base field $\K=\R$, there is the classical (and much more general) formality result in \cite{DGMS}.
Since we were not able to locate an extension of this result to nonzero characteristic cases in the literature, we provide a proof of Lem.~(\ref{formal})  at the end of this section.
In fact, the characteristic condition in the statement of Lem.~(\ref{formal}) can be removed by a straightforward modification of the proof.

\begin{proposition}\label{noqiso}
If the characteristic of the base field $\K$ is $2$, then the $A_\infty$-algebra $\Aff$ is not $A_\infty$ quasi-isomorphic to the algebra of cochains on a closed orientable surface $S$.
\end{proposition}
\begin{proof}

By Lem.~(\ref{formal}), it suffices to prove that there is no $A_\infty$ quasi-isomorphism between $\Aff$ and $H^*(S;\K)$.
Suppose that $\Fff$ is an $A_\infty$-algebra homomorphism from $\Aff$ to  $H^*(S; \K)$. 
Since $\Aff$ is a minimal $A_\infty$-algebra, $\mu^1_\Aff$ is trivial. 
We have also established above that $\mu^3_\Aff$ vanishes as well. 
As a consequence, a particular set of $A_\infty$-functor equations, satisfied by the family of graded multilinear maps $\Fff^d : \Aff^{\otimes d} \to H^*(S;\K)[1-d]$, simplifies to 
\begin{align*}
\Fff^1  \left( \mu^4_\Aff (b_i, b_j, b_k,b_l) \right)   
&=  \Fff^3 (b_i, b_j, b_k) \cup \Fff^1(b_l) +  \Fff^1(b_i) \cup \Fff^3 ( b_j, b_k, b_l) +  \Fff^2(b_i, b_j) \cup  \Fff^2 (  b_k, b_l)\\
&
 + \Fff^3  \left(\mu_\Aff^2 (b_i, b_j) , b_k , b_l\right) + \Fff^3\left(b_i, \mu_\Aff^2 (b_j , b_k ), b_l\right) 
+ \Fff^3\left(b_i, b_j , \mu_\Aff^2 (b_k , b_l)\right) 
\end{align*}

In the rest of the proof we refer to the above equation as $Eq_{(i,j,k,l)}$
and consider the sum of all the equations $Eq_{(i,j,k,l)}$ where $(i,j,k,l)$ is a cyclic permutation of $(1,2,5,4)$, $(1,5,2,4)$ or $(2,5,2,5)$. 

First of all, the computation preceding this proposition implies that the sum of the left hand side of these ten equations is equal to $(-1)^{p+1} \Fff^1(a_0)$.
In contrast, the right hand side of the sum of these ten equations is $0$. 
Once we establish this, we get $\Fff^1(a_0)=0$ and hence $\Fff$ is not a quasi-isomorphism.

To prove the vanishing of the right hand side we consider the terms on the right hand side in three separate groups and argue that each group adds up to $0$ under the assumption that $char(\K)=2$. First observe that, since the cup product $\cup$ is (graded-)commutative, each of the first two terms on the right hand side of the equation $Eq_{(i,j,k,l)}$ appears in exactly one other equation, namely $Eq_{(l,i,j,k)}$ or $Eq_{(j,k,l,i)}$ . 
For the same reason, the third term on the right hand side of $Eq_{(i,j,k,l)}$ is cancelled by that of $Eq_{(k,l, i,j)}$, unless of course $(i,j)=(k,l)$. 
This leaves us with the sum of the third terms of $Eq_{(2,5, 2,5)}$ and $Eq_{(5, 2,5,2)}$,
$$\Fff^2(b_2, b_5) \cup \Fff^2 (  b_2, b_5) + \Fff^2(b_5, b_2) \cup \Fff^2 (  b_5, b_2)$$
which vanish by the general properties of the cup product.
Finally, remember that we always have $\mu^2_\Aff  (b_i , b_j)=0$ for $|i-j|\neq 3$ and, if $char(\K)=2$, 
$$\mu^2_\Aff  (b_4 , b_1) = \mu^2_\Aff  (b_2 , b_5) = \mu^2_\Aff  (b_5 , b_2)=a_0 \ . $$
This suffices to conclude that each of the last three terms on the right hand side of any one of the equations is either 0 or it appears in exactly two of our equations, e.g. the fifth term in $Eq_{(2,4,1,5)}$ is equal to the fifth term in $Eq_{(2,5,2,5)}$ \ .
\end{proof}

\begin{corollary}\label{nofill}
The Legendrian knot $\Lambda_{p,q,r,s}$ admits an augmentation which is not induced by an exact orientable Lagrangian filling.
\end{corollary}

\begin{remark}\label{intr}
When $char (\K)\neq2$ our proof of Prop.~(\ref{noqiso}) breaks down because the right hand side of the sum of the ten equations we consider in the last step of the proof is equal to
$$ 2 (-1)^p \left( \Fff^3  (a_0 , b_2 , b_5) + \Fff^3(b_2, a_0, b_5) + \Fff^3(b_2, b_5 , a_0)\right)$$
which is not necessarily $0$ in general.
\end{remark}

\begin{proof} \emph{(of Lem.~(\ref{formal}))}
We prove the formality of the differential graded algebra $C=C^*(S, \K)$ of (simplicial) cochains with the cup product on the closed surface $S$ associated to the triangulation given in Fig.~(\ref{simplicial}) by providing a zig-zag of explicit dg-algebra quasi-isomorphisms connecting $C$ and the cohomology algebra $H=H^*(S, \K)$ of $S$. 

\begin{figure}[htb!]
\centering
\begin{tikzpicture}
	\draw (0,0) node[inner sep=0] {\includegraphics[scale=0.3]{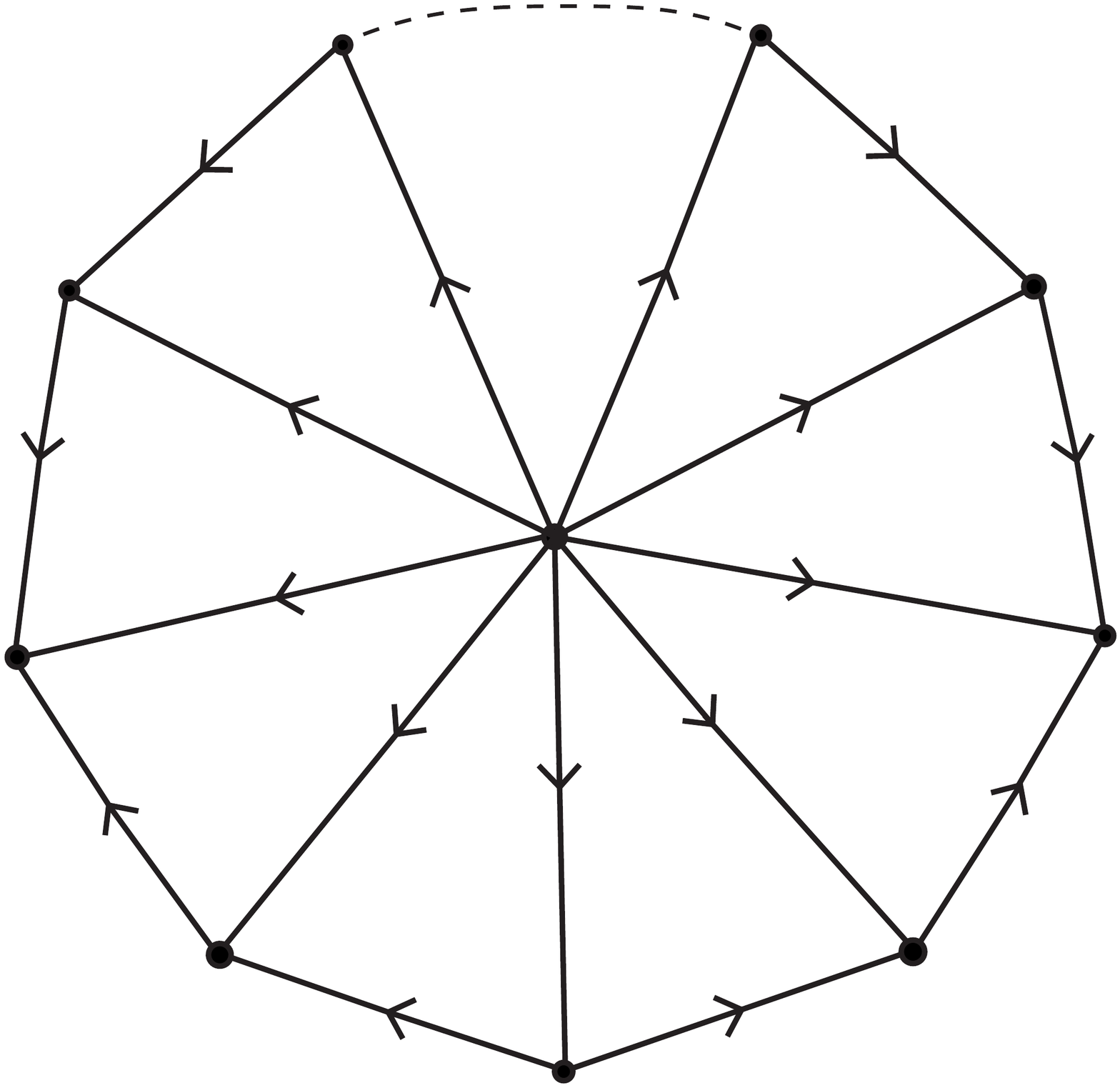}} ;
	
	\node[] at (0,1.05)  {{\small{$e^0_1$}}};
	\node[] at (0,-2.88)  {{\small{$e^0_2$}}};
	\node[] at (3,-1.2)  {{\small{$b_1$}}};
	\node[] at (2,3)  {{\small{$b_1$}}};
	\node[] at (3.3,1.1)  {{\small{$a_1$}}};
	\node[] at (1.1,-2.57)  {{\small{$a_1$}}};
	\node[] at (-1.1,-2.57)  {{\small{$b_g$}}};
	\node[] at (-3,-1.2)  {{\small{$a_g$}}};
	\node[] at (-3.3,1.1)  {{\small{$b_g$}}};
	\node[] at (-2,3)  {{\small{$a_g$}}};
	\node[] at (0.7,-1.2)  {{\small{$e^2_1$}}};
	\node[] at (-0.7,-1.2)  {{\small{$e^2_{4g}$}}};
	\node[] at (0.05,-1.75)  {{\small{$e^1_1$}}};
	\node[] at (1.7,-0.4)  {{\small{$e^2_2$}}};
	\node[] at (-1.7,-0.4)  {{\small{$e^2_{4g-1}$}}};
	\node[] at (1.5,-1.4)  {{\small{$e^1_2$}}};
	\node[] at (-1.5,-1.4)  {{\small{$e^1_{4g}$}}};

\end{tikzpicture}
\vspace{-1.2cm}
\caption{A triangulation of a closed, orientable surface $S$ of genus $g$}
\label{simplicial}
\end{figure}

We denote the generators of $C$ by
$$e_i, \alpha_j, \beta_j, \theta_k, \gamma_k \ \ \ \mbox{for } 1\leq j \leq g , \ 1\leq k \leq 4g $$
which represent the duals of the simplicies
$$e^0_i, a_j, b_j , e^1_k, e^2_k \, $$
as indicated in Fig.~(\ref{simplicial}). 
The nontrivial differentials and products can be read from the triangulation as
$$\d e_1  = \d e_2 = \theta_1 + \cdots + \theta_{4g} $$
$$ \d \alpha_j  = \gamma_{4j-3}+\gamma_{4j-1} , \ \ \d \beta_j  = \gamma_{4j-2}+\gamma_{4j} , \ \ 
\d \theta_k  = \gamma_{k-1} +\gamma_k $$
and
$$e_ie_i=e_i, \ \ e_1\theta_k = \theta_k , \ \ e_1\gamma_k = \gamma_k, \ \ \theta_k e_2 = \theta_k , \ \ \gamma_k e_2 = \gamma_k$$
$$ e_2 \alpha_j = \alpha_j e_2 = \alpha_j , \  \ e_2 \beta_j = \beta_j e_2 = \beta_j$$
$$\theta_{4j-3}\alpha_j=\gamma_{4j-3} , \ \ \theta_{4j-2}\beta_j=\gamma_{4j-2}$$
$$\theta_{4j}\alpha_j=\gamma_{4j-1} , \ \ \theta_{4j+1}\beta_j=\gamma_{4j}$$
(In the above equations and the rest of the proof, indices should always be interpreted modulo $4g$.)

We now define another dg-algebra, quasi-isomorphic to $C$ and with a simplified differential so that the rest of the proof is more transparent. 
This new dg-algebra $C'$ is generated by 
$$e,  \varphi_j, \psi_j, \nu, \ep_1, \zeta_1, \xi_l,  \nu_l   \ \ \ \mbox{for } 1\leq j \leq g , \ 1 \leq l \leq 4g-1$$ 
so that the map $\Phi : C' \to C$
defined by 
\begin{align*}
\Phi: \ \ \ \ \ & e \mapsto e_1 + e_2  ,  \ \ \ \ep_1 \mapsto e_1, \ \ \ \zeta_1 \mapsto \theta_1 + \cdots \theta_{4g} \\
&  \varphi_j \mapsto \alpha_j + \theta_{4j-2} + \theta_{4j-1},  \  \ \ \psi_j \mapsto \beta_j + \theta_{4j-1} +\theta_{4j} , \\ & \xi_l \mapsto \theta_l  ,  \ \  \nu_l \mapsto \gamma_{l-1} + \gamma_{l}, \ \ \ \nu \mapsto \gamma_{4g}
\end{align*}
is a dg-algebra quasi-isomorphism.
More precisely, on $C'$, the nontrivial differentials are 
$$\d' \ep_1 = \zeta_1, \ \ \d' \xi_l = \nu_l $$
$e$ is the identity element, and the remaining products are 
\begin{align*}
\varphi_j \psi_j = \nu + \nu_1 + \cdots + \nu_{4j-2} , \ & \ \psi_j \varphi_j = \nu + \nu_1 + \cdots + \nu_{4j-1}\\
\xi_{4j-3}\varphi_j=\nu + \nu_1 + \cdots + \nu_{4j-3} ,  \ & \ \xi_{4j-2}\psi_j=\nu + \nu_1 + \cdots + \nu_{4j-2}\\
\mbox{ for } j<g, \ \ \xi_{4j}\varphi_j=\nu + \nu_1 + \cdots + \nu_{4j-1},  \ & \   \xi_{4j+1}\psi_j=\nu + \nu_1 + \cdots + \nu_{4j} ,  \   \ \xi_{1}\psi_g=\nu\\
\ep_1\ep_1=\ep_1, \ \ \ep_1\xi_l = \xi_l , \ & \ \ep_1\varphi_j = \xi_{4j-2} +\xi_{4j-1}\\ 
 \mbox{ for } j<g,  \ \ \ep_1\psi_j = \xi_{4j-1} +\xi_{4j}, \ &  \ \ep_1\psi_g= \zeta_1 + \xi_1 + \cdots \xi_{4g-2}, \  \ \ep_1\nu = \nu \\
\ \ \ep_1\zeta_1 = \zeta_1, \ \ \ep_1\nu_l = \nu_l,\ & \  \zeta_1\varphi_j= \nu_{4j-2} + \nu_{4j-1} \\
\mbox{ for } j<g,  \ \ \zeta_1\psi_j=\nu_{4j-1}+\nu_{4j} , \ & \    \zeta_1\psi_g= \nu_1 + \cdots + \nu_{4g-2}
\end{align*}

In the next step, we define yet another dg-algebra $\widehat{C}$ by stabilizing $C'$, i.e. $\widehat{C}$ contains $C'$ as a subalgebra and the inclusion map is a dg-algebra quasi-isomorphism. 
Namely, we add the generators 
$$ \ep_{k}, \zeta_k  \ \ \ \mbox{for }  2\leq k \leq 2g+1$$ 
with $$|\ep_k|=0 , \ |\zeta_k|=1 \ \mbox{ and } \ \widehat{\d} \ep_{k} = \zeta_{k} , \ \widehat{\d} \zeta_{k} = 0$$
to those of $C'$.
and extend the algebra structure to $\widehat{C}$ by adding the following nontrivial products
\begin{align*}
\ep_{2j}\psi_j = \xi_1 + \cdots + \xi_{4j-2}, \ & \ \ep_{2j+1}\varphi_j = \xi_1 + \cdots + \xi_{4j-1} \\
\zeta_{2j}\psi_j = \nu_1 + \cdots + \nu_{4j-2}, \ & \ \zeta_{2j+1}\varphi_j = \nu_1 + \cdots + \nu_{4j-1}
\end{align*}
for $j=1, \dots , g$.

Finally, it is clear that the map $\widehat{\Phi} : H \to \widehat{C}$ defined on the cohomology algebra $H=H^*(S; \K)$ by 
$$\widehat{\Phi} : \ \ e \mapsto e , \ \ \overline{\varphi}_j \mapsto \varphi_j + \zeta_{2j} , \ \ \overline{\psi}_j \mapsto \psi_j + \zeta_{2j+1} , \ \ \nu \mapsto \nu$$
is a dg-algebra quasi-isomorphism
proving the formality of $C=C^*(S; \K)$.
\end{proof}

\end{document}